\documentclass[12pt,a4paper]{article}
\usepackage{amsmath,amssymb,amsthm,amsfonts,amscd,euscript,verbatim, t1enc, newlfont}
\usepackage{hyperref}

 \theoremstyle{plain}
\newtheorem{theo}{Theorem}[subsection]

\newtheorem{pr}[theo]{Proposition}

 \newtheorem{lem}[theo]{Lemma}
 \newtheorem{coro}[theo]{Corollary}
  
\theoremstyle{remark}
\newtheorem{rema}[theo]{Remark}

\theoremstyle{definition}
\newtheorem{defi}[theo]{Definition}
\newtheorem*{notat}{Notation}

 \newcommand\lan{\langle}
\newcommand\ra{\rangle}
\newcommand\bl{\bigl(} \newcommand\br{\bigl)}

\newcommand\smc{{SmCor}}
\newcommand\ssc{{Shv(SmCor)}}

\newcommand\dmge{DM^{eff}_{gm}{}}
\newcommand\dmgel{DM^{eff}_{gm,(l)}}
\newcommand\dmgep{DM^{eff}_{gm}[\frac{1}{p}]}

\newcommand\dmgm{DM_{gm}}
\newcommand\dmgml{DM_{gm,(l)}}
\newcommand\dmgmp{DM_{gm}[\frac{1}{p}]}

\newcommand\dme{DM_-^{eff}}
\newcommand\dmel{DM_{-,(l)}^{eff}}
\newcommand\dmep{DM_{-}^{eff}[\frac{1}{p}]}

\newcommand\mg{M_{gm}}
\newcommand\mgl{M_{gm,(l)}}
\newcommand\mgp{M_{gm}[\frac{1}{p}]}

\newcommand\obj{Obj}

\newcommand\id{id}
\newcommand\cu{\underline{C}}
\newcommand\du{\underline{D}}
\newcommand\eu{\underline{E}}
\newcommand\au{\underline{A}}

\newcommand\hw{{\underline{Hw}}}

\newcommand\z{{\mathbb{Z}}}
\newcommand\zll{{\mathbb{Z}_{(l)}}}
\newcommand\ql{{\mathbb{Q}_l}}
\newcommand\zl{{\mathbb{Z}_{l}}}

\newcommand\zop{{\mathbb{Z}[\frac{1}{p}]}}

\newcommand\q{{\mathbb{Q}}}
\newcommand\af{\mathbb{A}}
 
\newcommand\p{\mathbb{P}}

\newcommand\pt{pt}

\newcommand\ns{\{0\}}

\DeclareMathOperator\inli{\varinjlim}

\newcommand\chow{Chow}
\newcommand\chowp{Chow[\frac{1}{p}]}
\newcommand\chowe{Chow^{eff}}
\newcommand\chowel{Chow^{eff}_{(l)}}
\newcommand\chowep{Chow^{eff}[\frac{1}{p}]}

\newcommand\hip{HI [\frac{1}{p}]}
\newcommand\hil{HI_{(l)}}

\newcommand\tcho{t_{Chow}}
\newcommand\htcho{\underline{Ht}_{Chow}}

\newcommand\ab{Ab}
\newcommand\var{Var}
\newcommand\sv{SmVar}
\newcommand\spv{SmPrVar}

\newcommand\spe{\operatorname{Spec}\,}

 \DeclareMathOperator\ke{\operatorname{Ker}}
 \DeclareMathOperator\cok{\operatorname{Coker}}

\DeclareMathOperator\co{\operatorname{Cone}}

\DeclareMathOperator\cha{\operatorname{char}}
\DeclareMathOperator\adfu{\operatorname{AddFun}}

\begin{document}

 \title{
$\zop$-motivic resolution of singularities, and applications}
 \author{M.V. Bondarko
   \thanks{ 
The author acknowledges the support by  RFBR
(grant no. 10-01-00287a), by Saint-Petersburg State University research grant no. 6.38.75.2011,  by the Federal Targeted Programme "Scientific and Scientific-Pedagogical Personnel of the Innovative Russia in 2009-2013" (Contract No. 2010-1.1-111-128-033), by  Landau Network-Centro Volta and the Cariplo Foundation, and by the University of Milano.  } }
 \maketitle
\begin{abstract}
The main goal of this paper is to
deduce (from a recent resolution of singularities result
of Gabber) the following fact: (effective) Chow motives with $\zop$-coefficients
over a perfect field $k$ of characteristic $p$ generate
the category $\dmgep$ (of effective geometric Voevodsky's motives with $\zop$-coefficients). It follows
that $\dmgep$ 
can be endowed with a {\it Chow weight structure} $w_{\chow}$
whose {\it heart} is $\chowep$ (weight structures were introduced
in a preceding paper,  where the existence of $w_{\chow}$
for $\dmge\q$ was also proved). As shown in previous papers,
this statement immediately yields the existence of a conservative
 {\it weight complex} functor $\dmgep\to K^b (\chowep)$ (which induces an isomorphism of $K_0$-groups), as well as
 the existence of canonical and functorial (Chow)-weight spectral
 sequences and weight filtrations for any cohomology theory
 on $\dmgep$. We also define a certain {\it Chow} $t$-structure for $\dmep$ and relate it with unramified cohomology.
To this end we study {\it birational motives} and birational homotopy invariant sheaves with transfers.

\end{abstract}


\tableofcontents

 \section*{Introduction}

It is well known that Hironaka's resolution of singularities is very important for the theory of (Voevodsky's) motives over characteristic $0$ fields; see \cite{1}, \cite{gs}, and also \cite{mymot} and \cite{bws}. 

The purpose of this paper is to derive (as many as possible) 'motivic' consequences from the recent resolution of singularities result of Gabber (see Theorem 1.3 of \cite{illgab}). His result could be called '$\zll$-resolution of singularities' over a perfect characteristic $p$ field $k$ (where $l$ is any prime distinct from $p$). Surprisingly Gabber's theorem  is sufficient to extend all those properties of Voevodsky's motives (with integral coefficients, over characteristic $0$ fields) that were proved in \cite{bws}, to $\zop$-motives over $k$. In particular (in the notation of \S\ref{sbmot}) we prove the existence of a conservative exact {\it weight complex} functor $\dmgep\to K^b (\chowep)$, and that $K_0(\chowep)\cong K_0(\dmgep)$. We also establish the existence of  (Chow)-weight spectral sequences for any  cohomology theory defined on $\dmgep$ (those generalize Deligne's weight spectral sequences).

Previously the results mentioned were known to hold only for motives with rational coefficients (in  preceding papers we noted  that these rational coefficient versions can be proved using de Jong' s alterations, but did not give detailed  proofs). Since the results of this paper also hold for  motives with coefficients in any $\zop$-algebra, as 
 a by-product
we justify these claims (in more detail than before). 

Most of the results of this paper are already known  for $\cha k=0$ and motives (and cohomology) with integral coefficients.  Yet
we prove some  results on birational motives and birational sheaves (see \S\S\ref{sbir}--\ref{sunr}) that are partially new for this case also; note that our proofs  work (without any changes) in this alternative setting.


The central 'technical' notion of this paper is the one of {\it weight structure}. Weight structures are natural counterparts of $t$-structures for triangulated categories, 
introduced in \cite{bws} (and independently in \cite{konk}).  They were thoroughly studied and applied to motives in \cite{bws} and \cite{bger} (see also the survey preprint \cite{bsurv}).
Weight structures allow proving several properties of motives. In particular, most of the results mentioned above follow from the following (central) theorem: $\dmgep$ can be endowed with a weight structure $w_{\chow}$
whose {\it heart} is $\chowep$. The language of weight structures is also crucial for our proof of this statement (even though the main difficulty was to prove that $\chowep$ generates $\dmgep$ as a triangulated category). In contrast, note  that the methods of Gillet and Soul\'e (whose weight complex functor defined in \cite{gs} is  the 'first ancestor' of 'our weight complexes') only allow  proving the existence of weight complexes either with values in $K^b(\chowe\q)$ or in the category of unbounded complexes of $\zll$-Chow motives; cf. Remark \ref{rwc} below.

Now we list the contents of the paper. More details can be
found at the beginnings of
sections.

In the first section we recall some basic properties of motives and weight structures. Most of them are just modifications of some of the results of  \cite{1} and \cite{bws}; the only absolutely new result is a new condition for the existence of weight structures. We also recall a recent result on resolution of singularities over characteristic $p$ fields (proved by O. Gabber), and deduce certain (immediate) motivic consequences from it. 

In \S2 we prove our central theorem on the existence of the Chow weight structure for $\dmgep$; we deduce this  result from its certain $\zll$-version.

\S3 is dedicated to the applications of the central theorem  (yet we  deduce some of the results directly from the Gabber's one). We prove that the Chow weight structure can be extended to $\dmgmp$. It follows that $K_0(\dmgmp)\cong K_0(\chowp)$
 (and also $K_0(\dmgep)\cong K_0(\chowep)$). Also, there exists a conservative exact weight complex functor $\dmgmp\to K^b(\chowp)$ (which restricts to a functor $\dmgep\to K^b(\chowep)$).  The existence of the Chow weight structure also implies the existence of canonical $\dmgep$-functorial (starting from $E_2$) {\it Chow-weight spectral sequences} that express (any) cohomology of objects of $\dmgmp$ in terms of that of Chow motives. As was shown in \cite{bws}, these spectral sequences generalize the weight spectral sequences of Deligne (note that one can take any cohomology theory and $\zop$-coefficients here).

Next we prove that the Chow weight structure yields a weight structure for the category of {\it birational motives} i.e. for (the idempotent completion of) the localization of $\dmgep$ by $\dmgep(1)$ (see \cite{kabir}); its heart contains birational motives of all smooth varieties. We also study birational sheaves. Next we prove the existence of a certain {\it Chow} $t$-structure for $\dmep$ (whose heart is $\adfu(\chowep,\ab)$). Our results allow us to express unramified cohomology in terms  of the Chow $t$-structure cohomology of homotopy invariant sheaves with transfers.

Lastly, we recall that a method of M. Levine (described in \cite{hubka}, and combined with the fact that $\chowp$ generates $\dmgmp$) yields a perfect duality for $\dmgmp$; this allows  defining $\zop$-motives with compact support for arbitrary smooth varieties.

 The idea to write this paper was initiated by an interesting talk of M. Kerz at   the
conference
"Finiteness for motives and motivic cohomology"
 (Regensburg, 9--13th of February, 2009).  The author is
deeply grateful to prof. Uwe Jannsen
and to other organizers of
 this conference for their efforts, and to prof. D. H\'ebert, prof. D.-Ch. Cisinski, and prof. D. Rydh for their important comments.

\begin{notat}

 For a category $C,\ A,B\in\obj C$, we denote by
$C(A,B)$ the set of  $C$-morphisms from  $A$ to $B$.

For categories $C,D$ we write $D\subset C$ if $D$ is a full 
subcategory of $C$.

For a category $C,\ X,Y\in\obj C$, we say that $X$ is a {\it
retract} of $Y$ if $\id_X$ can be factorized through $Y$
(if $C$ is triangulated or abelian, then $X$ is a  retract of $Y$
whenever $X$ is its direct summand).

 For an additive $D\subset C$ the subcategory $D$ is called
{\it Karoubi-closed}
  in $C$ if it
contains all retracts  of its objects in $C$.
The full subcategory of $C$ whose objects are all retracts of objects of $D$ (in $C$) will be called the Karoubi-closure of $D$ in $C$.

$X\in \obj C$ will be called compact if the functor $ C(X,-)$
respects all  small coproducts that exist in $C$ (contrary
to  tradition, we do not assume that arbitrary coproducts exist).

For an additive $B$, $X,Y\in \obj B$, we will write $X\perp Y$ if $B(X,Y)=\ns$.
For $D,E\subset \obj B$ we will write $D\perp E$ if $X\perp Y$
 for all $X\in D,\ Y\in E$.
For $D\subset B$ we will denote by $D^\perp$ the class
$$\{Y\in \obj B:\ X\perp Y\ \forall X\in D\}.$$
Dually, ${}^\perp{}D$ is the class
$\{Y\in \obj B:\ Y\perp X\ \forall X\in D\}$.

$\cu$ below will always denote some triangulated category; usually it will
be endowed with a weight structure $w$ (see Definition \ref{dwstr}
below).

We will use the term 'exact functor' for a functor of
triangulated categories (i.e.  for a functor that preserves the
structures of triangulated categories). We will call a contravariant
additive functor $\cu\to \au$ for an abelian $\au$ {\it cohomological} if
it converts distinguished triangles into long exact sequences.

For $f\in\cu (X,Y)$, $X,Y\in\obj\cu$, we will call the third vertex
of (any) distinguished triangle $X\stackrel{f}{\to}Y\to Z$ a cone of
$f$; recall that different choices of cones are connected by
(non-unique) isomorphisms.

We will often specify a distinguished triangle by two of its
morphisms.

 For a set of
objects $C_i\in\obj\cu$, $i\in I$, we will denote by $\lan C_i\ra$
the smallest strictly full triangulated subcategory containing all $C_i$; for
$D\subset \cu$ we will write $\lan D\ra$ instead of $\lan C:\ C\in\obj
D\ra$.

We will say that some $C_i\in \obj\cu$ generate $\cu$ if $\cu$ equals $\lan
C_i\ra$. We will say that $C_i$ {\it weakly generate} $\cu$ if for
any $X\in\obj\cu$ such that $\cu(C_i[j],X)=\ns$ for all $i\in I,\
j\in\z$ we have $X=0$ (i.e. if $\{C_i[j]\}^\perp$ contains only
 zero objects).

 $D\subset \obj \cu$ will be
called {\it extension-stable}
    if for any distinguished triangle $A\to B\to C$
in $\cu$ we have: $A,C\in D\implies
B\in D$.

 $k$ will be our perfect base field of characteristic $p$ ($p$ will be positive everywhere except those places where we will explicitly specify the opposite).
 $\var\supset \sv\supset \spv$ will denote the set of all
varieties over $k$, resp. of smooth varieties, resp. of smooth
projective varieties.

$l$ below  will be some prime number distinct from $p$ (we will assume it to be fixed from time to time).
 
 \end{notat}

\section{Preliminaries: motives and weight structures}

In this section we recall some basics on motives, weight structures, and resolution of singularities.

In \S\ref{sbmot}  we study Voevodsky's motives with various coefficient rings (following \cite{vbook} and \cite{1}).  

In \S\ref{sgab} we recall a recent result of Gabber on resolution of singularities; we also 'translate it into a motivic form'.

In \S\ref{sbws} we recall those basics of the theory of weight structures (developed in \cite{bws}) that will be needed below. 

In \S\ref{slws} we prove  a certain new criterion for the existence of a weight structure in a certain situation. 

\subsection{Some basics on motives with various coefficient rings}\label{sbmot}

For motives with integral coefficients we use the
 notation of \cite{1}: $\smc$, $\ssc$ (the category of Nisnevich sheaves with transfers),
 $\chowe\subset \dmge\subset\dme\subset D^-(\ssc)$;
 $\mg:\sv\to \dmge$; $\z(1)$.

Now recall that (as
 was shown in \cite{vbook}), one can do the theory of motives with
 coefficients in an arbitrary commutative associative ring
with a unit $R$.
  One should start with the naturally defined category of
$R$-correspondences:
  $\obj(\smc_R)=\sv$; for $X,Y$ in $\sv$ we set
 $\smc_R (X,Y)=\bigoplus_U R$ for
all integral closed $U\subset X\times Y$  that are finite over $X$
and dominant over a connected component of $X$.
Proceeding as in \cite{1} (i.e. considering the corresponding localization of $K^b (\smc_R)$, and complexes of sheaves with transfers with homotopy invariant cohomology) one obtains the theory of
motives (i.e. of $\dmge{}_R$ that lies in $\dmgm{}_R$ and in $\dme{}_R$) that satisfies all basic properties of the
'usual' Voevodsky's motives (i.e. of those with integral coefficients; note that some of the results of \cite{1} were extended to the case   $\cha k>0$  in \cite{degdoc} and \cite{hubka}).
So we will apply these properties of motives with $R$-coefficients
without any further mention. 

In this paper we will mostly consider motives with $\zop$
and $\zll$-coefficients. We will denote
by $\chowep\subset \dmgep\subset\dmep$, $\mgp:\sv\to \dmgep$
(resp. $\chowel\subset \dmgel\subset\dmel$,  $\mgl:\sv\to \dmgep$)
the corresponding analogues of Voevodsky's notation (note that we have all of the full embeddings listed indeed).
We will also need $\chowp\subset \dmgmp$.

We list some of the properties of motivic complexes that we will need below.
Recall that $\dme$ supports the so-called homotopy $t$-structure $t$ (coming from $D^-(\ssc)$). 
The heart of  $t$ is the category $HI$ of homotopy invariant (Nisnevich) sheaves with transfers. 
Below we will denote  the hearts of the restrictions of $t$ to $\dmep\supset \dmel$ by $\hip\supset \hil$.

\begin{pr}\label{pbmot}
1. The
 functors $\dme\to \dmep$ (resp. $\dmep\to \dmel$) given by
 tensoring sheaves by $\zop$ (resp. $\zop$-module sheaves by $\zll$) tensor all morphism groups by $\zop$ (resp. by $\zll$).  The same is true for the (compatible) functors
 $\chowe\to \chowep\to \chowel$ and $\dmge\to \dmgep\to\dmgel$.
 
 2. The collection of functors $\otimes_{\zll}:\dmep\to \dmel$ for $l$ running through all primes $\neq p$, is conservative (on $\dmep$).
 
 3. The forgetful functors that send 
 a complex of $\zop$-module sheaves to the underlying complex of sheaves of abelian groups
  (resp. a complex of $\zll$-module sheaves to the underlying complex of $\zop$-module sheaves) yield full embeddings $\dmel\subset\dmep\subset \dme$.

 4. For any $U\in \sv$, $m\in \z$, $S\in \obj \dmep$ (resp. $S\in \obj \dmel$)  the $m$-th
  hypercohomology of $U$ with coefficients in $S$ is naturally isomorphic to $\dmep(\mgp(U),S[m])$ (resp. to $\dmel(\mgl(U),S[m])$).
  
 5. $t$  can  be restricted to $\dmep$ and $\dmel$; the two functors connecting $\dmep$ with $\dmel$ (described in the previous assertions) are $t$-exact with respect to these restrictions.

6. All objects
 of $\dmgep$ are compact in $\dmep$.

7. Let $f:U\to V$ be an open dense embedding of smooth varieties; let $S\in \obj HI$. Then $S(f)$ is injective.

8. For any $X\in \sv$ we have: $\dmep(X),\dmel(X)\in \dmep^{t\le 0}$. 
 
\end{pr}
\begin{proof}
1. It suffices to note that $\zop$ is flat over $\z$, and $\zll$ is flat over $\zop$.

2. Immediate from assertion 1.

3. Indeed, these functors are one-sided inverses of the functors $\dme\to \dmep\to\dmel$ described in assertion 1. 

4. Immediate from Proposition 3.2.3 and Theorem 3.2.6 of \cite{1}.

5,6. Easy from the previous assertions. 

7. Immediate from Corollary 4.19 of \cite{3}.

8. Immediate from the corresponding fact for $\mg(X)$, which  is obvious given  Proposition 3.2.6 of \cite{1}.

\end{proof}

 \begin{rema}

One can also easily see: all the results proved below for $\zop$-motives are also valid  for motives with coefficients in an arbitrary (unital commutative) $\zop$-algebra; to this end our proofs can be adjusted straightforwardly.
\end{rema}

Lastly, we note (though this will not be important at all below) that $\obj \chowel$ is (probably) larger than   $\obj \chowep$
(and than $\obj \chowe$) since when we increase the coefficient ring
we could get more idempotents; the same could happen for
$\obj \dmgep\subset \obj \dmgel$.


\subsection{Gabber's $\zll$-resolution of singularities}\label{sgab}

Let $l\neq p$ be fixed. The foundation of this paper is the following
result (which easily follows from a result of O. Gabber).

\begin{pr}\label{pgab}

For any $U\in \sv$ there exist an open dense subvariety $U'\subset U$
and a finite flat  
 morphism $f:P'\to U'$ (everywhere) of degree prime to $l$, for  $P'\in \sv$ such that  $P'$ has a smooth  projective compactification $P$.

\end{pr}
\begin{proof}

We can assume that $U$ is connected.

Let $Q'$ be some compactification of $U$. Then by Theorem 1.3 of
\cite{illgab} there exist a finite field extension $k'/k$ of degree prime to $l$ (it is separable since $k$ is perfect), 
 a smooth
quasi-projective $Q/k'$, and a 
finite 
surjective morphism
$g:Q\to Q'_{k'}$ of degree prime to $l$.
Since $g$ is proper, 
$Q$ is actually projective (in our case).
We can also assume that  $g_U$ is flat  (since we can replace $U$ by some $U''/k$). 

Now we  restrict scalars from $k'$ to $k$
and denote $Q$ considered as a variety over $k$ by $P$. We obtain that $P\in \spv$, and that there exists a
 finite flat morphism from some $P'\subset P$ to $U'\times \spe k'$; the degree of this morphism is prime to $l$. 
Lastly, it remains to compose this morphism with the natural morphism  $U'\times \spe k'\to U$, whose degree is also prime to $l$.

\end{proof}

Now we reformulate this statement 'motivically'.

\begin{coro}\label{cgab}

Let $U\in \sv$, $\dim U=m$.

1. For $U',P'$ as in Proposition \ref{pgab}, $\mgl(U')$ is a retract of $\mgl(P')$.

2. There also exist sequences $X_i,Y_i\in\obj \dmgel$, $0\le i\le m$, and $f_i\in \dmgel(X_{i},X_{i-1})$,  $g_i\in \dmgel(Y_{i},Y_{i-1})$ (for $1\le i\le m$), such that: $X_0=\mgl(U)$, $X_m=\mgl(U')$, $Y_0=\mgl(P)$, $X_m=\mgl(P')$, $\co f_i=\mgl(V_i)(i)[2i]$,  $\co g_i=\mgl(W_i)(i)[2i]$, for some smooth varieties $V_i,W_i/k$ of dimension $m-i$ (that could be empty).

\end{coro}

\begin{proof}

1. The transpose of the graph of $f$
yields a finite correspondence from   $U'$
to $P'$ (in the sense of \cite{1}). 
Composing it with  $f$ and considering as a morphism of motives, we obtain
$\deg f\cdot\id_{\mgl(U')}$ (see Lemma 2.3.5 of \cite{2}). Since $\deg f$ is prime to $l$, we obtain that
$\mgl(U')$ is a retract of  $\mgl(P')$ in $\dmgel$.

2. We recall the Gysin distinguished triangle (see Proposition 4.3 of
\cite{degdoc} that establishes its existence in the case $\cha k>0$). For a closed embedding $Z\to X$ of smooth varieties,
$Z$ is everywhere of
codimension $c$ in $X$, it has the form:
\begin{equation}\label{gys}
\mg(X\setminus Z)\to \mg(X)\to \mg(Z)(c)[2c]\to \mg(X\setminus Z)[1];
\end{equation}
certainly, obvious analogues  exist  for the functors $\mgp$ and
$\mgl$.

Hence in order to prove the assertion it suffices  to choose a sequence of $U_i,P_i\in \sv$ such that: $U_0=U'\subset U_1\subset U_2\subset \dots \subset U_m=U$ (resp. $P_0=P'\subset P_1\subset P_2\subset \dots \subset P_m=P$), $U_{i}\setminus U_{i-1}$ is non-singular and has codimension $i$ everywhere in $U_{i}$ (resp. $P_{i}\setminus P_{i-1}$ is non-singular and has codimension $i$ everywhere in $P_{i}$) for all $i$. Now, in order to obtain such $U_i$ and $P_i$ it suffices to consider stratifications of $U\setminus U'$ and $P\setminus P'$.

\end{proof}

\subsection{Weight structures: reminder}\label{sbws}

\begin{defi}\label{dwstr}
I A pair of subclasses $\cu^{w\le 0},\cu^{w\ge 0}\subset\obj \cu$
will be said to define a weight
structure $w$ for $\cu$ if 
they  satisfy the following conditions:

(i) $\cu^{w\ge 0},\cu^{w\le 0}$ are additive and Karoubi-closed
(i.e. contain all retracts of their objects that belong to
$\obj\cu$).

(ii) {\bf Semi-invariance with respect to translations.}

$\cu^{w\ge 0}\subset \cu^{w\ge 0}[1]$; $\cu^{w\le 0}[1]\subset
\cu^{w\le 0}$.

(iii) {\bf Orthogonality.}

$\cu^{w\ge 0}\perp \cu^{w\le 0}[1]$.

(iv) {\bf Weight decompositions}.

 For any $X\in\obj \cu$ there
exists a distinguished triangle

\begin{equation}\label{wd}
B[-1]\to X\to A\stackrel{f}{\to} B
\end{equation} such that $A\in \cu^{w\le 0}, B\in \cu^{w\ge 0}$.

II The full subcategory $\hw\subset \cu$ whose objects are
$\cu^{w=0}=\cu^{w\ge 0}\cap \cu^{w\le 0}$, 
 will be called the {\it heart} of 
$w$.

III $\cu^{w\ge i}$ (resp. $\cu^{w\le i}$, resp.
$\cu^{w= i}$) will denote $\cu^{w\ge
0}[-i]$ (resp. $\cu^{w\le 0}[-i]$, resp. $\cu^{w= 0}[-i]$).

IV We denote $\cu^{w\ge i}\cap \cu^{w\le j}$
by $\cu^{[i,j]}$ (so it equals $\ns$ for $i>j$).

V We will say that $(\cu,w)$ is {\it bounded above} if
$\cup_{i\in \z} \cu^{w\le i}=\obj \cu$.

VI We will  say that $(\cu,w)$ is {\it  bounded}  if
$\cup_{i\in \z} \cu^{w\le i}=\obj \cu=\cup_{i\in \z} \cu^{w\ge i}$.

VII Let $H$ be a 
full subcategory of a triangulated $\cu$.

We will say that $H$ is {\it negative} if
 $\obj H\perp (\cup_{i>0}\obj (H[i]))$.

VIII We will say that a triangulated category $\cu$ is {\it bounded with respect to} some $H\subset \obj \cu$
if
for any $X\in \obj\cu$ there exist  $j_X,q_X\in\z$ such that
\begin{equation}\label{bouab} \obj H\perp\{ X[i],\ i>q_X\}\text{ and } \{ X[i],\ i<j_X\}\perp \obj H.\end{equation}

 IX We call a category $\frac A B$ the {\it factor} of an additive
category $A$
by its (full) additive subcategory $B$ if $\obj \bl \frac A B\br=\obj
A$ and $(\frac A B)(X,Y)= A(X,Y)/(\sum_{Z\in \obj B} A(Z,Y) \circ
A(X,Z))$.

\end{defi}

\begin{rema}\label{rstws}

A  simple (and yet useful) example of a weight structure is given by the stupid
filtration of objects of
$K^b(B)\subset K(B)$ for an arbitrary additive category $B$. 
For this weight structure  $K(B)^{w\le 0}$ (resp. $K(B)^{w\ge 0}$) is the class of complexes that are
homotopy equivalent to complexes
 concentrated in degrees $\le 0$ (resp. $\ge 0$); below we will also need $K(B)^{[i,j]}$ 
 (as in  Definition \ref{dwstr}(IV)). The heart of this weight structure (either for $K(B)$ or for $K^b(B)$) 
is the   Karoubi-closure  of $B$
 in the corresponding category. So, it is equivalent to $B$ if the latter is idempotent complete. 
\end{rema}

Now we recall those properties of weight structures that
will be needed below (and that can be easily formulated), and prove a certain new assertion.
We will not mention more complicated matters (weight complexes, $K_0$, and weight spectral sequences)
 here; instead we will just formulate
the corresponding 'motivic' results below.

\begin{pr} \label{pbw}
Let $\cu$ be a triangulated category; $w$ will be a weight structure for
$\cu$ everywhere except assertions (\ref{igen}) and (\ref{iort}).

\begin{enumerate}

\item\label{iext} $\cu^{w\le 0}$, $\cu^{w\ge 0}$, and $\cu^{w=0}$
are extension-stable. 


\item\label{itriv} 
For any $q,r\in \z$, $X\in \cu^{[q,r]}$,  there exist $X^q\in \cu^{w=0}$ and $f\in \cu(X,X^q[-q])$ such that $\co f\in \cu^{w\ge q}$.

\item\label{ibougen} 
For any $i\le j\in \z$  we have: $\cu^{[i,j]}$ is the smallest extension-stable subclass of $\obj\cu$
containing $\cup_{i\le l\le j}\cu^{w=l}$.
In particular, if  $w$ (for $\cu$) is bounded, then  $\cu=\lan\hw\ra$.

\item\label{iidemp}
If $w$ is bounded, then it extends to a bounded weight structure for the idempotent completion of $\cu$. The heart of this new weight structure is the idempotent completion of $\hw$.

\item\label{iloc}

Let $\du\subset \cu$ be a 
triangulated subcategory of
$\cu$. Suppose
 that $w$ induces a weight structure on $\du$
(i.e. $\obj \du\cap \cu^{w\le
 0}$ and $\obj \du\cap \cu^{w\ge
 0}$ give a weight structure for $\du$);
we denote the heart of this weight structure  by $HD$.

 Then $w$ induces a weight structure on
 $\cu/\du$ (the localization i.e. the Verdier quotient of $\cu$
by $\du$) i.e.: the Karoubi-closures of $\cu^{w\le
 0}$ and $\cu^{w\ge
 0}$ (considered as classes of objects of $\cu/\du$) give a weight structure for $\cu/\du$
(note that $\obj \cu=\obj \cu/\du$).
The heart of the latter is the Karoubi-closure  of $\frac { \hw} {HD}$ in
$\cu/\du$.


If $(\cu,w)$ is bounded  then $\cu/\du$
also is.

\item \label{igen}

Let $\cu$ be triangulated and idempotent complete; let  $H\subset \obj \cu$ be negative and additive. Then
there exists a unique bounded weight structure $w$ on the Karoubi-closure $T$
of $\lan H\ra$ in $\cu$ such that $H\subset T^{w=0}$. Its heart is
the Karoubi-closure of 
$H$ in $\cu$. 

\item\label{iort}  Let $\du$ be a triangulated category that is weakly generated by some additive set $H\subset D$ of compact objects; suppose that there exists an extension-stable $D\subset \obj \du$ such that $H\cup D[1]\subset D$, and  
 arbitrary (small) coproducts exist in $D$. 
Denote by $H'$ the Karoubi-closure of
the category of all (small) coproducts of objects of $H$ in $\du$;
denote by $\eu$ the triangulated subcategory of $\du$ whose objects 
are characterized by the  following part of (\ref{bouab}): 
there exists a $q_Y\in \z$ such that 
$\obj H\perp\{ Y[i],\ i>q_Y\}$. 

Then there exists a
bounded above weight structure $w'$ for $\eu$ such that $\hw'= H'$.

Besides, a compact $X\in \obj \du$ belongs to $\eu^{[j,q]}$ (for $j\le q\in \z$) whenever it satisfies (\ref{bouab}) with $j_X=j$ and $q_X=q$.

\end{enumerate}
\end{pr}
\begin{proof}
\begin{enumerate}

\item  This is Proposition 1.3.3(3) of \cite{bws}.

\item Immediate from the distinguished triangle $A\to B\to X[1]$ and the previous assertion.

\item A weight decomposition of $X[q]$ yields a distinguished triangle
$X\to A'\stackrel{f'}{\to} B'\to X[1]$ for $A'\in \cu^{w\le q}$,
$B'\in \cu^{w\ge q}$. Assertion \ref{iext} implies that
$A'\in \cu^{w=q}$. Hence we can take $X^q=A'[q]$, $f=f'$.



\item 
Easy from Proposition 1.5.6(2) of ibid. 

\item This is Proposition 5.2.2 of ibid.

\item This is Proposition 8.1.1 of ibid.

\item By  Theorem 4.3.2(II1) of 
    ibid., there exists a unique weight structure on $\lan H\ra$
such that
    $D\subset \lan H\ra^{w=0}$. Next, Proposition 5.2.2 of ibid. yields
that $w$ can be extended to the whole $T$; along with Theorem 4.3.2(II2) of ibid. it also allows  calculating $T^{w=0}$
 in this case.

\item The existence of $w'$ is immediate from Theorem 4.3.2(III), version (ii), of ibid.
The second part of the assertion is given by part V2 of loc.cit. (cf. Definition 4.2.1 of ibid.).

\end{enumerate}

\end{proof}

\subsection{The 'main weight structure lemma'}\label{slws}

The main part of the proof of the central theorem is a certain weight
structure statement (not contained in \cite{bws}). We formulate and prove it here, since it could be
used independently from motives (so it could be useful even if in the
future the resolution of singularities will be fully established over
fields of arbitrary characteristic).

\begin{pr} \label{pmwsl}

Let $\du,D,H$ be as in Proposition \ref{pbw}(\ref{iort}).
Let $\cu\subset\du$ be an idempotent complete triangulated subcategory such that all objects of
$\cu$ are compact in $\du$, $H\subset \cu$, and $\cu$ is bounded with respect to $H$. 

Then the following statements are valid.

1. $\cu$ is contained in the Karoubi-closure $I$ of $\lan H\ra$ in $\du$.

2. There exists a bounded weight structure $w$ for $\cu$ such that $\hw$ is
the Karoubi-closure of 
$H$ in $\cu$.

3. For $X\in \obj \cu$, we have: $X\in \cu^{[j,q]}$ whenever one can take $j$ for $j_X$ and $q$ for $q_X$ in (\ref{bouab}).

\end{pr}
\begin{proof}

We adopt the notation of Proposition \ref{pbw}(\ref{iort}).

We have $\cu\subset \eu$ (by the definition of the latter). Besides (as proved in loc.cit) the analogue of
assertion 3 with $w'$ instead of  $w$ and with $\eu^{[j,q]}$
instead of $\cu^{[j,q]}$
is valid.

Now we prove assertion 1. We denote $\obj I$ by $G$.

We should prove that
\begin{equation}\label{cucag}X\in \obj\cu \cap \eu^{[q,r]}\implies X\in G\end{equation} for any $q\le r\in \z$.


First let $q=r$. Then $X[q]$ is a retract of $\coprod_{i\in I} H_i$
for some set $I$ and $H_i\in \obj H$. So, $\id_{X[q]}$ factorizes
through $\coprod_{i\in I} H_i$. Since $X[q]$ is compact,
$\du(X[q],\coprod H_i)=\bigoplus \du(X[q],H_i)$; so
$\id_{X[q]}$ also can be factorized through $\coprod_{i\in J} H_i$
for some finite $J\subset I$. Hence $X[q]$ is a retract of
$\coprod_{i\in J} H_i$; so $X\in G$.

Now we prove (\ref{cucag}) in the general case by induction on $r-q$.

Suppose that it  is fulfilled for all $q,r$ such that $r-q\le m$ for
some $m\ge 0$. We prove
(\ref{cucag})
 for  some fixed $X\in \obj\cu \cap \eu^{[s,t]}$, where  $t-s=m+1$. By  Proposition \ref{pbw}(\ref{itriv}), there exist $X^s\in \obj H'$
 and $f\in \du(X,X^s[-s])$ such that $\co f\in \eu^{w'\ge s}$. By the
 definition of $H'$, $X^s$ is a retract of some $\coprod_{i\in I}H_i$,
 $H_i\in \obj H$. Since $\co f\in \eu^{w'\ge s}$, a cone of the
 induced morphism $X\to \coprod_{i\in I}H_i[-s]$ also belongs to
 $\eu^{w'\ge s}$ (since it is the direct sum of $\co f$ with the
 'complement' of $X^s[-s]$ to $\coprod_{i\in I}H_i[-s]$). So, we
 assume that $X^s=\coprod_{i\in I}H_i$. Now, since
 $\du(X,\coprod H_i[-s])=\bigoplus \du(X,H_i[-s])$, $f$ can be
 factorized through
 $\coprod_{i\in J} H_i[-s]$ (for some finite $J$). Then
 $\co f=\co(f':X\to \bigoplus_{i\in J} H_i[-s])\bigoplus\coprod_{i\in I\setminus J}X_i[-s]$, where $f'$ is the morphism 'induced' by $f$.
 So, $\co f'\in \eu^{w'\ge s}$; it also belongs to
 $\eu^{w'\le t}$ by  Proposition \ref{pbw}(\ref{iext}).
 Hence $\co f'\in G$. Since $\bigoplus_{i\in J} H_i[-s]\in G$,
 we obtain that $X\in G$.

  Now, Proposition \ref{pbw}(\ref{igen}) 
 implies that  $w'$
 can be restricted to $\cu$ and the weight structure $w$ obtained is
 the one required for assertion 2. 
 Besides, the reasoning above also proves assertion 3 (by Proposition \ref{pbw}(\ref{iext})).
 

\end{proof}

\section{Motivic resolution of singularities }


In \S\ref{szll} we prove 'almost a $\zll$-version' of our main
result. Then Proposition \ref{pmwsl} allows us to deduce our central theorem (in \S\ref{szop}).

\subsection{$\zll$-version of the central theorem}\label{szll}

We fix some $l(\in \p\setminus\{p\})$.

We prove a statement that is essentially the $\zll$-version of our main
result. We do not formulate it this way since our goal is just to
prepare for the proof of Theorem \ref{tres}. Yet the notation
$\dmgel{}^{[0,m]}$ certainly comes from weight structures.

\begin{pr}\label{pres}

1. $\dmgel$ is the idempotent completion of
$\lan \mgl(P),\ P\in \spv \ra$.

2. Let $U\in \sv$, $\dim U=m$; let $P\in \spv$. Then \break 
$\dmel(\mgl(U),\mgl(P)[i])=\ns$ for $i>0$;
$\dmel(\mgl(P),\mgl(U)[i])=\ns$ for $i>m$.

\end{pr}
\begin{proof}

First we note that by Theorem 5.23 of \cite{degdoc} the subcategory
$H_{\dmge}$ of $\dmge$ whose objects are $\{\mgl(P),\ P\in \spv\}$
is negative (here we use the isomorphism of $\dmge(\mg(X,\z(i)[j]))$ with the corresponding
higher Chow groups). Hence $\{\mgl(P),\ P\in \spv\}$ is negative in
$\dmgel$ also; we denote this category by $H$.

We define 
$\dmgel{}^{[0,r]}\subset \obj \dmgel$ for $r\ge 0$
as the smallest extension-stable Karoubi-closed subclass of $\obj \dmgel$
 that contains $\mgl(P)[-s]$ for all $P\in \spv$, $0\le s\le r$.

Since $\dmgel$ is the idempotent completion of
$\lan \mgl(U),\ U\in \sv\ra$ (in $\dmel$) by definition,
in order to prove assertion 1 it suffices to verify:  in $\dmel$ the
Karoubi-closure of $\lan \mgl(P),\ P\in \spv \ra$ contains all $\mgl(U)$
for $U\in \sv$.
Hence the negativity of $H$ easily implies: in order to prove both of our
assertions it suffices to verify that $\mgl(U)\in \dmgel{}^{[0,m]}$ for
any $U$ as in assertion 2.

The latter statement is obvious for $m=0$. We prove it in general by induction
on $m$.

First we note that $\dmgel{}^{[0,m]}(1)[2]\subset \dmgel{}^{[0,m]}$ for
any $m$, since $\mgl(P)(1)[2]$ is a retract of $\mgl(P\times \p^1)$
(for $P\in \sv$). Hence $\mgl(Z)(c)[2c]\in \dmgel{}^{[0,n-1]}$ for any $Z$ of dimension
$<n$ and any $c\ge 0$.

Suppose now that our assertion is true for all $m<n$  for some $n>0$.
We verify it for some $U$ of dimension $n$.

We apply Corollary \ref{cgab}(2). In the notation of loc.cit. (for $m=n$), we obtain  for any $i>0$: $X_{i-1}\in \dmgel{}^{[0,n]}$ whenever  $X_{i}\in \dmgel{}^{[0,n]}$, and $Y_{i-1}\in \dmgel{}^{[0,n]}$ whenever  $Y_{i}\in \dmgel{}^{[0,n]}$. Since $Y_0\in \dmgel{}^{[0,n]}$, the same is true for $Y_n$, hence also for $X_n$ and for $X_0=\mgl(U)$.

\end{proof}

\subsection{The main result: 'motivic $\zop$-resolution of singularities'}\label{szop}

\begin{theo}\label{tres}

1. $\dmgep$ is the idempotent completion of $\lan \mgp(P),\ P\in \spv \ra$.

2. There exists a bounded weight structure $w_{\chow}$ for $\dmgep$
such that $\hw_{\chow}=\chowep$.

3. For $U\in \sv$, $\dim U=m$, we have: $\mgp(U)\in \dmgep^{[0,m]}$.

4. For any open dense embedding $U\to V$, for $U,V\in \sv$, we have: $\co(\mg(U)\to \mg(V))\in \dmgep^{w_{\chow}\ge 0}$.

\end{theo}
\begin{proof}

We set $H=\{\mgp(P),\ P\in \spv \}$, $\cu=\dmgep$, and $\du=\dmep$, $D=\dmep^{t\le 0}$,
and verify that the assumptions of Proposition \ref{pmwsl} are fulfilled.

By Proposition \ref{pbmot}(6), all objects of $\dmgep$ are compact
in $\dmep$. We have $H\subset D$ by part 8 of loc.cit. 
Besides, $D$ is extension-stable, contains $D[1]=\dmep^{t\le -1}$, and admits arbitrary coproducts.

 Using Theorem 5.23 of \cite{degdoc} we obtain (similarly to the proof of Proposition \ref{pres}) that $H$ is
 negative.

By Proposition \ref{pres},
 for any $l(\neq p)$ the image of $\dmgep$ in $\dmgel$ is bounded with respect
 to the image of $H$ in $\dmgel$ 
  (one can easily deduce this
 fact from any of the parts of the proposition). Hence $\dmgep$ is
 bounded with respect to $H$.

 It remains to verify that for any $S\in \obj \dmel$, $S\neq 0$,
 there exist  $P\in\spv$ and $j\in \z$ such that
 $\dmep(\mgp(P),S[j])\neq \ns$.

 Recall that $\dmep$ is a full subcategory of $D^-(\ssc)$. 
 So there exist some $U\in\sv$ and $m\in \z$ such that the $m$-th
  hypercohomology of $S$ at $U$ is non-zero. We choose some
  $l\neq p$ such that this hypercohomology group is not $l$-torsion.
  Then the $m$-th
  hypercohomology at $U$ of $S_l$ is non-zero also, where $S_l$ is the
  image of $S$ in $\dmel$.
  Now, by  Proposition \ref{pbmot}(4)
  this
  group is exactly $\dmel(\mgl(U),S_l[m])$. Then  Proposition
  \ref{pres}(1) easily implies: there exist  $P\in\spv$ and $j\in \z$
  such that $\dmel(\mgl(P),S_l[j])\neq \ns$. Hence
  $\dmep(\mgp(P),S[j])\neq \ns$ also.

  Now we can apply Proposition \ref{pmwsl}; it yields assertions
  1 and 2 immediately. Applying    Proposition \ref{pres}(2)
  for all $l\neq p$ simultaneously along with Proposition \ref{pmwsl}(3),
  we prove assertion 3.
  
  Assertion 4 can be easily deduced from  assertion 3 by induction. To this end we choose a sequence of $U_i\in \sv$ such that: $U_0=U\subset U_1\subset U_2\subset \dots U_m=V$ (for some $m\in \z$) and $U_{i+1}\setminus U_i$ is non-singular and has some codimension $c_i$ everywhere in $U_{i+1}$ for all $i$. Then applying (\ref{gys}) repeatedly we obtain the result; cf. the proof of Proposition \ref{pres}.

\end{proof}

\begin{rema}\label{rres}

1. Our 'globalization' argument (i.e. passing from $\zll$-coefficients
to $\zop$-ones) certainly can be applied in other situations; it only
requires some of 'formal' properties of motives (with $\zop$ and
$\zll$-coefficients) to be fulfilled.

Moreover, one could even  pass to integral coefficients if a similar $\z_{(p)}$-information is available also.

2. A category of relative Voevodsky's
motives could be an example of  a setup of this sort.  This means: one should consider (some) Voevodsky's motives over a base scheme $S$; note that in \cite{degcis} a rational coefficient version of such a category was thoroughly studied and called the category of Beilinson motives, whereas in \cite{hebpo} and \cite{brelmot} a certain Chow weight structure for this category was introduced. Unfortunately, currently we don't know much about $S$-motives with $\zll$-coefficients.

3. We will deduce several implications from our Theorem below.
Now we will only note that any $X\in \obj\dmgep$ has a 'filtration'
(that can be easily described in terms of weight decompositions of
$X[i]$, $i\in \z$) whose 'factors' are objects of $\chowep$
(this is a {\it weight Postnikov tower} of $X$; see Definition 1.5.8 of
\cite{bws}). In particular, it follows that for any $U\in \sv$,
$X=\mgp(U)$, there exist an $X^0\in \obj\chowep$ and an
$f\in \dmgep (X,X^0)$ such that $\co f\in  \dmgep^{w_{\chow} \ge 0}$.
Note here that $\dmgep (X,X^0)$ can be described in terms of $\smc$;
one can assume that $X^0=\mgp(P)$ for some $P\in \spv$.

Now, if $U$ admits a smooth compactification $P$, then $\mgp(P)$
is one of the possible choices of $X^0$ (see part 4 of the theorem). So, our results yield the existence of  a certain
'motivic' analogue of a smooth compactification of $U$; this justifies
the title of the paper.
Moreover, for motives with $\zll$-coefficients one could try to find
some $X^0$ using Gabber's resolution of singularities of results. Yet
with $\zop$-coefficients this result seems to be very far from being
obvious from 'geometry'; it is also not clear how to look for a
'geometric' candidate for $X^0$ in the absence of a $\zop$-analogue of
Proposition \ref{pgab}.
\end{rema}

\section{Applications}

In \S\ref{sdmgm} we prove that the Chow weight structure can be extended to $\dmgmp$. We also compute certain $K_0$-groups of $\dmgep$ and $\dmgmp$.

In \S\ref{swc} we recall (following \cite{bws}) that the existence of $w_{\chow}$ implies the existence of the weight complex functor ($\dmgmp\to K^b(\chowp))$; it is exact and conservative), and of Chow-weight spectral sequences for any cohomology of motives. 

In \S\ref{sbir} we study {\it birational motives} and birational homotopy invariant sheaves with transfers (as defined in \cite{kabir}). Our results immediately yield the existence of a weight structure for $\zop$-birational motives whose heart contains all 'birational motives of smooth varieties'. This extends some results of 
ibid. to $\zop$-motives over $k$.

In \S\ref{sunr} we prove the existence of a certain {\it Chow $t$-structure} $\tcho$ for $\dmep$ whose heart is $\adfu(\chowep^{op},\ab)$. It turns out that a homotopy invariant sheaf with transfers $S$ belongs to the heart of $\tcho$ whenever it is birational. Moreover, $H^0_{\tcho}(S)$ is the largest birational subsheaf of $S$. Using this fact, we express unramified cohomology in terms of $\tcho$.

In \S\ref{smgc}  we prove that $\dmgmp$ is a perfect triangulated category: this follows easily  from the fact that this category is generated by $\chowp$ via a method of M. Levine and \cite{hubka}. It follows that for any smooth variety there exists a 'reasonable' motif with compact support for it (in $\dmgep$).

\subsection{The Chow weight structure for $\dmgmp$; $K_0$ for $\dmgep\subset \dmgmp$}\label{sdmgm}

Similarly to $\dmgm$ (as in \cite{1}) we define $\dmgmp$ as
$\dmgep[\zop(-1)]$, where $\mgp(\p^1)=\mgp(\pt)\bigoplus \zop(1)[2]$ (i.e. we invert $\zop(1)$ formally).

\begin{pr}\label{pwcgm}
1. $\dmgmp=\lan \chowp \ra$.

2. There exists a weight structure on $\dmgmp$ extending
$w_{\chow}$ for $\dmgep$, whose heart is $\chowp$.

3. We have $\dmgmp^{w_{\chow}\le 0}\otimes \dmgmp^{w_{\chow}\le 0}\subset \dmgmp^{w_{\chow}\le 0}$
and $\dmgmp^{w_{\chow}\ge 0}\otimes \dmgmp^{w_{\chow}\ge 0}\subset \dmgmp^{w_{\chow}\ge 0}$.

\end{pr}
\begin{proof}

Proposition \ref{pbw}(\ref{ibougen})
yields that $\dmgep=\lan \chowep \ra$. 
We deduce assertion 1 immediately.

Since $-\otimes \z(1)[2]$ is a full embedding of $\dmge$ into itself (see
\cite{voevc}), the same is true for $\dmgep$. Hence
$\chowp=\chowep[\zop(-1)[-2]]$ is negative in $\dmgmp$. Hence Proposition \ref{pbw}(\ref{ibougen}, \ref{igen})
along with assertion 1 implies assertions 2 and 3.

\end{proof}

\begin{rema}\label{rttw}
By assertion 3, for $X\in \obj \chowp\subset \obj \dmgmp$
the functor $-\otimes X$ is weight-exact i.e. it sends $\dmgmp^{w_{\chow}\le 0}$
and $\dmgmp^{w_{\chow}\ge 0}$ to themselves. In particular, this is true for
$X=\zop(1)[2]$. Moreover, since $-\otimes \zop(1)[2]$ is an invertible
functor, for any $i,j\in \z$ we have
$Y(1)[2] \in\dmgmp^{[i,j]}\iff  Y \in\dmgmp^{[i,j]}$.

\end{rema}

Now we calculate certain $K_0$-groups of $\dmgep\subset \dmgmp$.

\begin{pr}\label{pkz}

We define $K_0(\chowep)$ (resp.  $K_0(\chowp)$) as the groups whose
generators are $[X]$, $X\in \obj \chowep$ (resp. $X\in \obj \chowp$),
and the relations are: $[Z]=[X]+[Y]$ for
$X,Y,Z\in\obj \chowep$ (resp. $X,Y,Z\in \obj \chowp$) such that $Z\cong X\bigoplus Y$.
For $K_0(\dmgep)$ (resp. $K_0(\dmgmp)$)
we take similar generators and  set $[B]=[A]+[C]$ if
$A\to B\to C\to A[1]$ is a distinguished triangle.

Then the embeddings $\chowep\to \dmgep$ and $\chowp\to \dmgmp$ yield
isomorphisms $K_0(\chowep)\cong K_0(\dmgep)$ and
$K_0(\chowp)\cong K_0(\dmgmp)$.
\end{pr}
\begin{proof}
Immediate from Proposition \ref{pwcgm} and 
Proposition 5.3.3(3) of \cite{bws}.

Here we use the fact that $\dmgmp$ is idempotent complete since $\dmgep$ is.

\end{proof}

\begin{rema}
Certainly, we have similar isomorphisms for $\zll$-motives (as well
as for motives with coefficients in any commutative $\zop$-algebra).
Besides, all these isomorphisms are actually ring isomorphisms.
\end{rema}

\subsection{Weight complexes and weight spectral sequences for $\zop$-Voevodsky's motives}\label{swc}

We prove that the weight complex functor (whose 'first ancestor'
was defined by Gillet and Soul\'e)
can be defined for $\zop$-Voevodsky's motives.

\begin{pr}\label{pwc}
1. There exists an exact conservative weight complex functor
$t:\dmgmp\to K^b(\chowp)$ which restricts to an (exact conservative)
functor $\dmgep\to K^b(\chowep)$.

2. For $X\in \obj \dmgmp$, $i,j\in \z$, we have $X\in \dmgmp^{[i,j]}$ whenever  $t(X)\in K(\chowp)^{[i,j]}$ (see Remark \ref{rstws}). 
\end{pr}
\begin{proof}
1. By  Proposition 5.3.3 of \cite{bws}, this follows from the existence of bounded Chow weight structures for $\dmgep\subset \dmgmp$ along with the fact that these categories admit differential graded enhancements (see Definition 6.1.2 and \S7.3 of ibid.).

2. Immediate from  Theorem 3.3.1(IV) of ibid.
\end{proof}

\begin{rema}\label{rwc}

1. One can easily describe $t(\mgp(U))$ if $U\in \sv$ is the complement
of a normal crossings divisor to a smooth projective variety. To this end
one could apply the results of \S6.5 of \cite{mymot} along with
Poincare duality.

Now, similarly to Remark \ref{rres}(2), for a general
$U\in\sv$ one could try to calculate $t(\mgl(U))$ using Theorem 1.3
of \cite{illgab}. Yet $t(\mgp(U))$ seems to be rather mysterious from
the 'geometric' point of view.

2. The 'first ancestor' of  weight complex functors (the 'current' one and that for general triangulated categories with weight structures were introduced in \cite{bws}) was defined  in \cite{gs}. To a variety $X$ over a characteristic $0$ field they (essentially) assigned $t(\mg^c(X))$; see \S\S6.5-6.6 of \cite{mymot} and \S\ref{smgc} below.  Yet for $\cha k>0$ their methods only yield the existence of weight complexes  with values either in $K^b(\chowe\q)$ or in $K(\chowel)$ (i.e. they do not prove that $\zll$-weight complexes are always homotopy equivalent to bounded ones; see \S5 of \cite{sg}).

3. In \cite{mymot} in the case $\cha k=0$ also a certain differential graded 'description' of $\dmge$ was given (it is somewhat similar to the definition of Hanamura's motives; a comparison (anti)isomorphism from Voevodsky's $\dmgm$ to the category of Hanamura's motives was also constructed there). Unfortunately, this result relies heavily on certain consequences of  '$cdh$-descent', and it seems that no substitute for it is known in the case $\cha k>0$ (even for motives with rational coefficients).

\end{rema}


Now we discuss (Chow)-weight spectral sequences for cohomology of $\zop$-motives. One can also easily dualize this to obtain similar results for homological functors (see Theorem 2.3.2 of \cite{bws}). We note that any weight structure yields certain weight spectral sequences for any cohomology theory; the main difference of the result below from Theorem 2.4.2 of ibid. is that $T(H,X)$ always converges  (since our Chow weight structure is bounded).

\begin{pr}
Let $\au$ be an abelian category, $X\in \obj\dmgmp$; we denote by $(X^i)$ the terms of $t(X)$ (so $X^i\in \obj \chowp$; here we can take any possible choice of $t(X)$ as an object of $C^b(\chowp)$).

 I Let $H:\dmgep\to \au$ be a cohomological functor, $X\in \obj \dmgep$, $H^i=H([-i])$ for any $i\in \z$.
Then there exists a spectral sequence $T=T(H,X)$ with $E_1^{pq}=
H^{q}(X^{-p})\implies H^{p+q}(X)$; the differentials for $E_1^{**}(T(H,X))$ come from $t(X)$.
  
$T(H,X)$ is $\dmgep$-functorial in $X$ starting from $E_2$.

II Similar statements hold for any cohomological functor $H:\dmgmp\to \au$ (and any $X\in \obj\dmgmp$).
\end{pr}

\begin{proof}

Immediate from Theorem 2.4.2 of \cite{bws}. 

\end{proof}

\begin{rema}

1. The {\it Chow-weight} spectral sequence $T(H,X)$  induces a certain (Chow)-weight filtration on $H^*(X)$. This filtration is $\dmgep$-functorial (since $E_2(T)$ is). This filtration  can also be (easily) described in terms of weight decompositions (only); see \S2.1 of ibid.

2. We obtain certain (Chow)-weight spectral sequences and weight
filtrations for all realizations of motives. In particular, we have them
for \'etale cohomology of motives, and for
$\zop$-motivic cohomology.

Note here: it certainly suffices to have the Chow weight structure for $\dmgml$ in order to have Chow-weight spectral sequences for $H\otimes \zll$; yet without a $\zop$-weight structure it would not be clear at all that the whole collection of these spectral sequences (for all $l\neq p$) can be chosen to come from a single {\it weight Postnikov tower}  for $X$ (see Definition 1.5.8 of ibid.). In particular, it is not (really) important whether we use the $\zop$-Chow weight structure or the $\zll$-one in order to construct the weight spectral sequences for $\zl$-\'etale cohomology if we  fix $l$; yet $\zop$-weight structure yields certain 'relations' between these spectral sequences for various $l$, as well as with $\zop$-motivic cohomology. 

Recall also (see Remark  2.4.3 of of ibid.) that  
 the $\ql$-\'etale cohomology of motives the  weight filtration obtained coincides with the usual one (up to a shift of
indices). Besides,  note  that
'classically' the weight filtration (for \'etale cohomology) is well-defined only for
  rational (i.e. $\ql$-) coefficients. 

Lastly, recall that for motivic cohomology we obtain quite new spectral sequences (yet a certain easy partial case can be obtained from Bloch's long exact
localization sequence for higher Chow groups of varieties), that do not have to degenerate at any fixed level (even  rationally; see loc.cit.).


3. Certain  {\it weight spectral sequences} considered in \S2 of \cite{janhasse} are (essentially) examples of  Chow-weight spectral sequences. The author strongly suspects that some of the results of ibid. could be re-proved and extended using our methods.  

\end{rema}

\subsection{On birational motives}\label{sbir}

Now we prove that our methods  easily yield certain
properties of birational motives and sheaves (some of them were already proved in 
\cite{kabir}; yet note that we extend them to motives with $\zop$-coefficients for $\cha k=p$).
 
We define $\dmgmp^0$ as the idempotent completion of the
localization of $\dmgep$ by $\dmgep(1)=\dmgep\otimes \zop(1)$.
$\dmgmp^0$ is called the category of birational motives since
$\dmgep(1)$ is exactly the triangulated category generated by
$\co(\mgp(U)\to \mgp(X))$ for $U,X\in \sv$, $U$ is dense in $X$.
Indeed, this statement follows easily from (\ref{gys}) (and was
proved in Proposition 5.2 of ibid. in detail).

For the full embedding of categories $\chowep(1)[2]\subset\chowep$
we consider the fraction category $\frac{\chowep}{\chowep(1)[2]}$
defined via Definition \ref{dwstr}(IX); 
$\chowp^0$ is
its idempotent completion.

\begin{pr} 1. There exists a bounded weight structure
$w_{bir}$ for $\dmgmp^0$ whose heart is $\chowp^0$.

2. The image of $\mgp(X)$ in $\dmgmp^0$ belongs to $\hw_{bir}$ for any $X\in \sv$.
\end{pr}
\begin{proof}

1. Immediate from  Proposition \ref{pbw}(\ref{iloc}--\ref{iidemp}).

2. Let $H$ be the class of  images of $\mgp(X)$, $X\in \sv$, in $\dmgmp^0$. We prove that $H$ is  negative (in $\dmgmp^0$). To this end it obviously suffices to prove the natural analogues of this statement for $\dmgml^0$ (for all $l\neq p$). Then  Corollary \ref{cgab}(2) implies: it suffices to prove negativity for the images of $\mgp(P)$, $X\in \spv$ (in $\dmgmp^0$).
Hence the result  follows from assertion 1.

Proposition \ref{pbw}(\ref{igen}) yields: there exists a weight structure for $\dmgmp^0$ whose heart contains $H$. Since this heart also contains $\chowp^0$, we obtain that this new weight structure is exactly $w_{bir}$ (by the uniqueness of the weight structure given by loc.cit.). Hence $H\subset \dmgmp^0{}^{w_{bir}=0}$.

See also  Remark 4.9.2(2) of \cite{bger} for an alternative proof. 
\end{proof}

\begin{rema}

1. One of the main consequences of assertion 1 is the calculation of $\dmgmp^0(X,Y[i])$ for $X,Y\in \obj\chowp^0(\subset \obj \dmgmp^0)$, $i\ge 0$. 

2. Certainly, the same method works if $\cha k=0$; then one can
take integral coefficients.

3. We also obtain a conservative weight complex functor
$\dmgmp^0\to K^b(\chowep^0)$ and an isomorphism $K_0(\chowep^0)\to
K_0(\dmgmp^0)$.

\end{rema}

Below we will also need {\it birational sheaves}. The following statements could probably be proved using weight structures; yet  'sheaf-theoretic' proofs are easier. The proof of assertion I1 was (essentially) copied from   \S7 of \cite{kabir}. 

\begin{lem}\label{lbir}
I Let $S\in \obj HI\subset \obj \dme$.

1.  Let $S$ be {\it birational} i.e.  suppose that $S(f)$ is an isomorphism for
 any open dense embedding $f$ in $\sv$. Then $\dme(\mg(U),S[i])=\ns$ for any $U\in \sv$, $i>0$.

 2. $S$ is birational whenever
$\dme(X(1),S)=\ns$ for any $X\in \obj \dmge$.

II 1. The category $\hip_{bir}$ of birational $\zop$-module sheaves is an exact abelian subcategory of $\hip$.

2. Let $S\in \obj \hip$,  $S^0\in \obj \hip_{bir}$, $f\in \hip(S^0,S)$.  Then $f$ is a monomorphism whenever $f(P):S^0(P)\to S(P)$ is injective for any $P\in \spv$.
 
 3. $f:S\to S'$ is an isomorphism for $S,S\in \obj \hip_{bir}$ whenever $f(P)$ is bijective for any $P\in \spv$.

\end{lem}
\begin{proof}

I1. Since $S$ is birational, it is locally constant in the Zariski  topology (on $\sv$); hence it has trivial higher Zariski cohomology. Since $S$ is homotopy invariant, we obtain the same vanishing for Nisnevich cohomology by Theorem 5.7 of
\cite{3}. It remains to apply Proposition \ref{pbmot}(4).

2. Let $S$ satisfy the second condition. Then (\ref{gys}) yields that $S(f)$ is an isomorphism if $V\setminus U$ is smooth and everywhere of codimension $c$ in $V$ (for $f:U\to V$). Since any open embedding can be factored as the composition of embeddings satisfying this condition, we obtain that $S$ is birational.  

Conversely, let $S$ be birational. It suffices  to prove that $\dme(\mg(U)(1),S[i])=\ns$ for any $U\in \sv$, $i\in \z$. We have: $\mg(U\times \af^1)\cong \mg(U)$, $\mg(U\times G_m)=\mg(U)\bigoplus \mg(U)(1)$. We obtain: $$\begin{gathered} \dme(\mg(U)(1),S[i])\cong  \\ \cok( \dme(\mg(U\times \af^1),S[i+1])\to \dme(\mg(U\times G_m),S[i+1])).\end{gathered}$$ Applying  Proposition \ref{pbmot}(4), we obtain that this kernel is zero: for $i+1<0$ since sheaves have no negative cohomology; for $i+1=0$ since $S$ is birational, and for $i+1>0$ by assertion I1.

II 1. The kernel of a morphism of birational sheaves  is obviously birational. Next, the presheaf cokernel of such a morphism is a birational presheaf; hence it is a locally constant Zariski sheaf. Since it is also a homotopy invariant presheaf with transfers, we obtain that it belongs to $\obj \hip$ by Proposition 5.5 of \cite{3}; so it is a birational object of $\hip$. 

Lastly, an extension of birational sheaves yields a long exact sequence of their cohomology groups (at any section). Hence assertion I1 yields that such an extension is also an extension of presheaves; so it is obviously birational.

2. If $f$ is monomorphic, it is injective at all sections.

Now we prove the converse statement. It suffices to check it for $S$ and $S^0$ replaced by $S\otimes \zll$ and $S^0\otimes \zll$ (for all $l$); so we can assume that $S,S^0\in \obj \hil$. We fix some $l$.

We should check that 
 $f(U)$ yields an injection $S^0(U)\to S(U)$ for any $U\in \sv$.
 
 We fix some $U$ and apply  Corollary \ref{cgab}. In the notation of loc.cit., we have a commutative diagram $$\begin{CD}
S^0(P)@>{g}>>S^0(P')\\
@VV{h}V@VV{i}V \\
S(P)@>{j}>>S(P')\end{CD}$$
$g$ is bijective since $S_0$ is birational; $h$ is injective by our assumption; $j$ is injective by Proposition \ref{pbmot}(7);
 hence $i$ is injective also.

Since $S(U')$ is a retract of $S(P')$ and the same is true for  $S^0$, we obtain a similar injection for $U'$.
We have a diagram
$$\begin{CD}
S^0(U)@>{a}>>S^0(U')\\
@VV{b}V@VV{c}V \\
S(U)@>{d}>>S(U')\end{CD}$$
Now,  $d$ is injective, $a$ is bijective.  Since $c$ is injective, $b$ is injective also.

3. If sheaves are isomorphic, all their sections are isomorphic also.

Conversely, let $f(P)$ be an isomorphisms for any $P\in \spv$. By  Proposition \ref{pbmot}(4) and assertion I1 we obtain that $f_*:\dmep(\mgp(P)[i],S)\to \dmep(\mgp(P)[i],S')$ is bijective for any $i\in \z$ and $P\in \spv$. Then  Theorem \ref{tres}(1) yields that $S(U)\cong S'(U)$ for any $U\in \sv$.

\end{proof}

\subsection{$\tcho$ and unramified cohomology} \label{sunr}

We prove that $\dmep$ supports a certain {\it Chow} $t$-structure. Below $\htcho$ will denote its heart; $H^j_{\tcho}(Y)$ (resp. $H^j_{t}(Y)$) for $j\in \z$, $Y\in \obj\dmep$ will denote the $j$-th cohomology of $Y$ with respect to $\tcho$ (resp. with respect to $t$); so  $H^j_{\tcho}(Y)\in \obj \htcho\subset\obj\dmep$.

Note also:  Lemma \ref{lbir}(II1)  implies that any sheaf $S\in\obj \hip$ has a maximal birational subsheaf (since any two birational subsheaves of $S$ are subobjects of some single  birational subsheaf of $S$).

\begin{pr}\label{ptcho}
1. There exists a $t$-structure $\tcho$ for $\chowep$ whose heart is isomorphic to $\adfu(\chowep^{op}, \ab)$; this isomorphism is given by restricting  $\dmep(-,Y)$ to $\chowep\subset \dmep$ for $Y\in \obj \htcho\subset \obj \dmep$.

2. Let $\dmep^{\tcho\ge 0}$ (resp. $\dmep^{\tcho\le 0}$) denote the 'non-negative' (resp. 'non-positive') part of $\tcho$. Then we have $\dmep(X,S)=\ns$ if either $X\in \dmgep^{w_{\chow}\le 0}$ and $S\in \dmep^{\tcho\ge 0}[-1]$, or $X\in \dmgep^{w_{\chow}\ge 0}$ and $S\in \dmep^{\tcho\le 0}[1]$.

3. $\dmep^{t\ge 0}\subset \dmep^{\tcho\ge 0}$.

4.  $\dmep^{\tcho\le 0}\subset \dmep^{t\le 0}$.

5. $S\in \hip$ belongs to $\htcho$ whenever it is  birational in the sense of Lemma \ref{lbir}. 

6. For any $S\in \hip$ we have: $S^0=H^0_{\tcho}S$ is the maximal birational subsheaf of $S$ (in $\hip$). 

If $V\in \sv$ possesses a smooth projective compactification $P$, then the image of $S^0(V)$ in $S(V)$ equals the image of $S(P)$ in $S(V)$.

\end{pr}
\begin{proof}
1, 2. $\chowep$ weakly generates $\dmep$ by  Theorem \ref{tres}(1)
 (cf. also the proof of loc.cit.). Now the assertions are immediate from  Theorem 4.5.2(I1) of \cite{bws}.
 
 3. Obvious from assertion 1.
 
 4. Immediate from assertion 2 (since for any $t$-structure $t'$  for $\cu$ we have $\cu^{t'\le 0}=\cu^{t'\ge 0}{}^\perp$).
 
 5. Let $S\in \hip\cap \obj \htcho$. We should prove that for $f:U\to V$ being an open dense embedding in $\sv$ the  map $S(f)$ is bijective. Since $S\in \hip$, $S(f)$ is an injection by  Proposition \ref{pbmot}(7). On the other hand, by  Proposition \ref{pbmot}(4) we have an exact (in the middle) sequence  $S(V)\to S(U)\to \dmep(\co(\mg(U)\to\mg(V)),S[1])$. Now, $\co(\mg(U)\to\mg(V))\in \dmgep^{w_{\chow}\ge 0}$ by  Theorem \ref{tres}(4); hence $\dmep(\co(\mg(U)\to\mg(V)),S[1])=\ns$ by assertion 2. We obtain that $S(f)$ is also surjective.

 Conversely, let $S\in \hip$ be birational. By  Lemma \ref{lbir}(1),  $\dmep(\mgp(P),S[i])=\ns$ for any $i>0$, $P\in\spv$. Then assertion 1 implies that $S\in \dmep^{\tcho\le 0}$. It remains to note that $S\in \dmep^{\tcho\ge 0}$ by assertion 3.

 6.  First we prove that $\dmep(\mgp(Z)(j)[i],S^0)=\ns$ if $i>0$ or $j>0$, $Z\in \sv$, by induction on $\dim Z+j$. Obviously, it suffices to prove all $\zll$-analogues of this statement: we fix some $l$.
 
 For  $\dim Z=0$ the assumption is obvious. Now suppose that for $S\in \hil$ we have 
 $\dmel(\mgl(Z)(j)[i],S^0)=\ns$ if $i$ or $j$ is $>0$ and  $\dim Z+j< r$ (for some $r\ge 0$). We verify this equality for $Z=U$,  $U\in \sv$, $\dim U+j=r$.
 
 First suppose that $U\in \spv$. 
 Since $\chowel(j)[2j]\subset \chowel$,
   by the definition of $S^0$ we have $\dmel(\mgl(U)(j)[i],S^0)=0$ for $i\neq 2j$. It remains to consider the case $i=2j>0$. We use the fact that $\mgl(\af^j\times U)=\mgl(U)$ and $\mgl(U\times \p^1)=\mgl(U)\bigoplus \mgl(U)(1)[2]$. It follows that $$\begin{aligned} \dmel(\mgl(U)(j)[2j],S^0)=\dmel(\mgl(U)(j)[2j],S)\\ \subset \ke (S(U\times (\p^1)^j)\to S(U\times \af^j)).\end{aligned}$$ Now, this kernel is zero  by  Proposition \ref{pbmot}(7).

 It remains to apply Corollary \ref{cgab}(2). Since  our assumption 
 is valid for $Z=P$, it is also true for $Z=P'$ in the notation of loc.cit.;
 here we use the fact that $\dmel(-,S^0)$ converts distinguished triangles in $\dmgel$ into long exact sequences. Since $S^0$ is also additive, we obtain the assumption for $Z=U'$, and hence also for $Z=U$. Our assumption is proved.
 
 We deduce that $S\in \dmep^{t\ge 0}$. Since it also belongs to $\htcho$; it is a birational sheaf by assertions 3 and 5.
 
 Now, for any $P\in \spv$ we have $S^0(P)\cong S(P)$ by the definition of $S^0$. Hence $S^0$ is a subsheaf of $S$ by  Lemma \ref{lbir}(II2). We also obtain the second half of the assertion.

 We denote the maximal birational subsheaf of $S$ by $S'$. Then $S^0$ is also a subsheaf of $S'$. We immediately obtain that $S^0(P)\cong S'(P)$ for any $P\in \spv$. Hence  loc.cit. allows us to conclude the proof.

\end{proof}

Now we relate the Chow $t$-structure with unramified cohomology; cf. 2.2 of \cite{merkunram}. Let $C\in \obj \dmep$.
Recall that the 
$i$-th unramified cohomology of $X\in \sv$ with coefficients in $C$ (we denote it by $H^i_{un}(X,C)$) is the intersection of images $H^i(\spe A,C)\to H^i(\spe k(X),C)$, where $A$ runs through all discrete valuation subrings of $k(X)$. 
Here we define the cohomology of 'infinite intersections' of smooth varieties as the corresponding inductive limits. We note here that any 
geometric valuation (of rank $1$) of a function field $K/k$ comes from a non-empty smooth subscheme of some  smooth variety  $U$ such that $k(U)=K$, 
 since the singular locus of any normal variety has codimension $\ge 2$.

\begin{pr}\label{punram}
For any $X,C$ as above there is a natural isomorphism $H^i_{un}(X,C) \cong H^0_{\tcho}(H^i_t(C))(X)$.

\end{pr}
\begin{proof}
We can obviously assume that $i=0$. Moreover, we can (and will) also assume that $C=H^0_t(C)$, since for any smooth semi-local $U$ (in the sense of \S4.4 of \cite{3})) we have $C(U)\cong  H^0_t(C)(U)$ by Lemma 4.28 of ibid. 
Hence $C$ yields a cycle module in the sense of Rost (see \cite{degcyc}). 

We denote $H^0_{\tcho}(C)$ by $C^0$. By Proposition \ref{ptcho}(6), 
$C^0$ is a birational subsheaf of $C$. We should prove that $s\in C(\spe k(X))$ comes from all $C(\spe A)$ whenever it belongs to $C^0(\spe k(X))$. 

Applying $C$ to (\ref{gys}) and passing to the inductive limit we obtain a 
 long exact sequence  $\ns\to C(\spe A)\to C(\spe k(X))\to C((\spe K)(1)[1])\to \dots$.
 Here $K$ is the residue field of $A$, and we define $C((\spe K)(1)[1])=\inli \dmep(\mg(U)(1)[1],C)$ for $U$ running through all smooth varieties with $k(U)=K$.
 
 Hence we should find out which $s$ vanish in all $C(\spe K(1)[1])$. If $s\in C^0(\spe k(X))$ then it vanishes in $C^0(\spe K(1)[1])$ since $C^0$ is birational; hence the image of $s$ in $C(\spe K(1)[1])$ is zero also. 

It remains to prove that for any $s\notin C^0(\spe k(X))$ there exists an $A$ such that the image of $s$ in (the corresponding)  $C(\spe k(X))$ is non-zero.

First we prove this statement for all 
 $X$ that possess a smooth projective compactification $P$. 
 By Proposition \ref{ptcho},   $C^0(\spe k(X))$ is the image of  $C(P)$ in $C(\spe k(X))$.
Besides, $C(P)$ is exactly the unramified cohomology group in question (see \S2.3 of \cite{merkunram}). Hence such an $A$ exists in this case.

Now we prove our assertion in the general case. It obviously suffices to prove it for $C\otimes \zll$ for all $l\neq p$. We fix some $l$.

By Proposition \ref{pgab} there exists a (finite) extension $L$ of $k(X)$ of degree prime to
to $l$ such that $L=k(P)$ for some $P\in \spv$. Considering the trace of $L/k(X)$ (divided by $\deg (L/k(X))$ we obtain that $C\otimes \zll(\spe k(X))$ is a retract of  $C\otimes \zll(\spe L)$ (we define the latter similarly to $C(\spe k(X))$). Hence there exists a discrete valuation ring $A'\subset L,\ L=kA$, such that the image of $s$ in  $C\otimes \zll(\spe L)$ does not come from $C\otimes \zll(\spe A')$. Then $A=A'\cap k(X)$ is a discrete valuation ring also, and $s\otimes 1$ does not come from $C\otimes \zll(\spe A)$. The proof is finished.

\end{proof}

\begin{rema}
Actually, one can generalize the proposition to the calculation of unramified cohomology with coefficients in any cohomology theory $\dmge\to \au$, where $\au$ is an abelian category satisfying AB5. To this end one should replace the corresponding $t$-truncations of $C$ by  {\it virtual $t$-truncations} (of the cohomological functor 'represented' by $C$) with respect to the {\it Gersten} and Chow weight structures (for {\it comotives}; all of the notions mentioned were defined and studied in \cite{bger}).  
Yet such a generalization would be somewhat 'tautological'.
\end{rema}

\subsection{Duality in $\dmgep$; motives with compact support}\label{smgc}

Applying an argument of Levine described in Appendix B of \cite{hubka},
we obtain that the full subcategory of $\dmgmp$ generated by $\chowp$
(i.e. the whole $\dmgmp$) enjoys a perfect duality such that the dual of
$\mgp(P)$ for $P\in \spv$ is $\mgp(P)(-m)[-2m]$ if $P$ is purely of
dimension $m$.

The only original statement that we will prove here is the following one.

\begin{pr}\label{pdual}
The dual of $\dmgmp^{w_{\chow}\le 0}$ with respect to this duality is
$\dmgmp^{w_{\chow}\ge 0}$, and vice versa.
\end{pr}
\begin{proof}

Immediate from Proposition \ref{pwc} and 
Proposition \ref{pbw}(\ref{ibougen}); note that this duality respects distinguished triangles and Chow motives.
\end{proof}

\begin{rema}
\label{rdual}
1. Certainly, proposition \ref{rdual}  implies that
$\widehat{\dmgmp^{[i,j]}}=\dmgmp^{[-j,-i]}$.

2. As explained in Appendix B of \cite{hubka}, using duality one can
define reasonable motives with compact support over $k$: for
$U\in\sv$  purely of dimension $m$ we set
$\mgp^c(U)=\widehat{\mgp(U)}(m)[2m]\in \obj \dmgep$.

So, we have $\mg^c(U)\in \dmgep^{[-\dim U,0]}$.


\end{rema}


\end{document}